\documentclass[12pt]{amsart}
\usepackage{latexsym} 
\usepackage{amsmath,amssymb,amsfonts,amsthm}
\usepackage{verbatim}
\usepackage[dvips]{graphicx}

\newtheorem{theorem}{Theorem}[section]
\newtheorem{lemma}{Lemma}[section]
\newtheorem{definition}{Definition}[section]
\newtheorem{proposition}[theorem]{Proposition}
\newtheorem{corollary}[theorem]{Corollary}

\newcommand{\R }{ \ensuremath{ \mathbf{R} }}

\newcommand{\dive}{\mathrm{div}_E\, }
\newcommand{\grade}{\mathrm{grad}_E }
\newcommand{\grad}{\mathrm{grad}\,}
\newcommand{\hess}{\mathrm{Hess}}
\newcommand{\tr}{\mathrm{Tr}}

\theoremstyle{definition}

\newcommand{\eop }{ \hfill $\Box$ }
\newcommand{\mr }{ \mathbf{R}}
\newcommand{\mP }{\mathbf{P}}
\newcommand{\me}{\mathbf{E}}

\title{Foliated Stochastic Calculus: Harmonic measures}

\thanks{Research partially supported by FAPESP 11/50151-0 and 12/18780-0.}

\keywords{Foliation, diffusion process, stochastic calculus. \\
\indent 2010 {\it Mathematics Subject Classification.  Primary:
58J65, 53C12 ; Secondary: 60H30, 60J60}}

\author{Pedro J. Catuogno, Diego S. Ledesma and Paulo R. Ruffino}
\address{Departamento de
 Matem\'{a}tica, Universidade Estadual de Campinas,\\ 13.081-970 -
 Campinas - SP, Brazil.}
\email{pedrojc@ime.unicamp.br ; dledesma@ime.unicamp.br and 
ruffino@ime.unicamp.br}
\begin{document}

\begin{abstract}
\noindent In this article we present an intrinsic construction of foliated
Brownian motion (FoBM) via stochastic calculus adapted to a foliated Riemaniann
manifold $(M, \mathcal{F})$.
The stochastic
approach together with this proposed foliated vector calculus provide a natural
method to work with (L. Garnett's) harmonic measures in $M$. New results include,
beside
an explicit stochastic equation for the FoBM, a decomposition of
the Laplacian of $M$ in terms of the foliated and the basic Laplacians, a 
characterization
of totally invariant measures and a differential equation for the density of
harmonic measures.
\end{abstract}
\maketitle

\section{Introduction}

Harmonic measures in a foliated Riemannian manifold $M$ are invariant measures
for
foliated Brownian motion (FoBM), which is a diffusion
associated to the foliated Laplacian $\Delta_E$. A Borel probability measure
$\mu$ is harmonic if for any bounded measurable function $f$ which is leafwise
smooth,
\[
\int_M \Delta_Ef~d\mu=0.
\]
Harmonic measures have a central place in
the ergodic
theory of foliations and in the study of asymptotic properties of the leaves.
Originally introduced by L. Garnett \cite{Garnett}, this concept lays in the
intersection
of many different areas like stochastic analysis, differential geometry and
functional analysis. 
See e.g. Candel \cite{candel1}, Kaimanovich \cite{kaimainovich}, Yue \cite{yue},
Adams
\cite{adams}, Ghys \cite{ghys1}, Ledrappier \cite{ledrappier}, Alcalde Cuesta
and Rechman \cite{Alcalde-PJM}, \cite{Alcalde-DCDS}, Bonatti et al.
\cite{Bonatti} and others.

The usual construction of the heat semigroup associated
to the Brownian motion on the leaves use an 
analytic functional approach. In this article we apply stochastic calculus 
to construct explicitly an equation for the diffusion processes
associated to this
semigroup, Theorem 3.4. Among other advantages of this approach, 
the problem of continuity is automatically solved (see, \cite[p. 188]{candel1}).
Moreover, with this equation, many aspects of the dynamics of FoBM can be considered, in this article we focus mainly on  harmonic measures.

Initially, in the next section, we
introduce the foliated operators and study
properties of foliated Laplacians. In particular, Theorem 2.1 presents an
alternative interpretation of the foliated Laplacian in terms of the classical
Laplacian in $M$, the basic Laplacian (Tondeur \cite{Tondeur}) and a geometrical
vector field  $\kappa$.

In Section 3  we study the foliated Brownian motion. An stochastic equation for this process (Theorem 3.4) is 
obtained using the geometrical technique of projecting onto the leaves a
horizontal process in the
orthonormal frame bundle $O(E)$ of the leaves. Here we are inspired by
Eells-Elworthy-Malliavin argument for Riemannian manifolds (e.g.
\cite{Elworthy}, \cite{IW}, among others).

The main results of this paper are in Section 4 where we study the harmonic
measures. We start with totally invariant measures, which locally can be described as the product of holonomy
invariant measures and the volume measure on the leaves. It is well known in
the literature that their existence  depends on the geometry of the foliation
(see e.g. \cite{Garnett}, \cite{candel}). L. Garnett \cite{Garnett} has showed
that these measures are harmonic. In
Theorem 4.3 we prove a characterization of
totally invariant measures in terms of the foliated divergent. Corollary 4.4
again characterizes these measures in terms of the usual divergent in $M$ and
the geometrical vector field $\kappa$.


Equations for the density of invariant measures
has a natural interest in this context. 
See e.g. Bogashev et. al. \cite{Bogashev1} for the elliptic
case or the corresponding evolution
equation \cite{Bogashev2}. With this sort of applications in mind, we
present a differential  equation for the
density of harmonic measures in Theorem 4.6. In particular, this equation
provides a tool for studying the support of harmonic measures. Many examples are explored along the text.

\section{The foliated Laplacian}

In this section we  introduce the fundamental operators which are structural 
in the theory of foliated spaces. Our framework here is a Riemannian manifold
$(M,g)$ which is foliated by a family of submanifolds $\mathcal{F}$ which is
characterized by the integrable
distribution $E \subseteq TM$ given by the tangent bundle of the leaves. Let
$\pi:TM\rightarrow
E$ be the orthogonal projection on $E$, naturally the metric $g$
induces metrics $g_E$  and $g^\perp$ in $E$ and  $E^{\perp}$ respectively.



Let $\nabla$ be the Levi-Civita connection in $M$ with respect to $g$. Denote by
$\nabla^E$
the connection on $E$ induced by $\nabla$, i.e. 
\[
\nabla^E_XY=\pi\nabla_XY
\]
for all $X\in TM$ and $Y$ in $\Gamma ( E )$, the space of smooth sections of $E$
over $M$. The connection $\nabla^E$ is the Levi-Civita connection on the leaves
with respect to $g_E$.




\begin{definition}\label{z2}
Let $f: M \rightarrow \mr$ be a smooth function and $X,\: Y$ sections of $E$. We
define the foliated operators: 
\begin{itemize}
\item[a)] $\grade\:f=\pi (\grad f)$;

\vspace{3mm}

\item[b)] $\dive Y=\mathrm{Tr}_E\, g(\nabla_\cdot^E Y,\, \cdot\, )$, where
$\mathrm{Tr}_E\,$ is the trace on $E$;

\vspace{3mm}

\item[c)]$\mathrm{Hess}_E (f)(X,Y)=XY(f)-\nabla^E_XYf$;

\vspace{3mm}

\item[d)]$\Delta_Ef=\mathrm{div}_E(\mathrm{grad}_E f)=\mathrm{Tr}_E
\mathrm{Hess}_Ef$.
\end{itemize}
\end{definition}
Extending vector fields and smooth functions from a leaf to the manifold $M$,
one sees that the operators above
are the natural extension of
 the corresponding operator on the leaf.

Given $\{X_1,\ldots,X_p\}$ a local orthonormal basis of $E$, the following
classical formulae hold:
\begin{itemize}
\item[a)] $\grade\:f=\sum_{i=1}^p(X_if)\:X_i,$

\vspace{3mm}

\item[b)] $\dive Y=\sum_{i=1}^p \left(g(\nabla^E_{X_i}Y,X_i)\right),$

\vspace{3mm}

\item[c)] $\Delta_Ef=\sum_{i=1}^p\hess_E f(X_i,X_i)$.
\end{itemize}

Comparing the Hessian on $M$ with the foliated operator $\hess_E$, we have that
\begin{equation} \label{omega}
 \hess f(X,Y) = \hess_E f(X,Y) - W(X,Y)f
\end{equation}
where $W(X,Y)$ is the second fundamental form of the foliation. Also, by this
formula one finds that
\begin{eqnarray}\label{z3}
\Delta_Ef=\tr_E\, (\hess f)+K f,
\end{eqnarray}
where $K$ is the  mean curvature of the foliation, defined as
$K=\tr_E W$.

Given smooth functions  $f$ and $h$, the classical formulae below hold: 
\begin{equation}
 \label{z4}
\hess_E(fh) =
h\hess_Ef+f\hess_Eh+dh|_E\otimes
df|_E+df|_E\otimes dh|_E
\end{equation}
and
\begin{equation} \label{z4.1}
\Delta_E (fh) = f\Delta_E h+h\Delta_E f+2\: g(\grade {f},\grade h).
\end{equation}



 Following the definition presented by Rumler  \cite{Rumler} and \cite{Rumler1}
we introduce the characteristic form of $\mathcal{F}$.

\begin{definition}
Let $E$ be an orientable bundle. The characteristic form of
$\mathcal{F}$, denoted by $\chi_{E}$, is the differential $p$-form
on $M$ defined by
\[
\chi_{E}(Y_1,\ldots,Y_p)=\det[g(Y_i,E_j)]
\]
where $Y_1,\ldots,Y_p\in TM$ and $\{E_1,\ldots,E_p\}$ is a local
positively oriented orthonormal basis of sections of $E$.
\end{definition}

 The restriction of $\chi_E$ to tangent vectors to a leaf $L \in \mathcal{F}$ is
an induced volume form
in $L$. The characteristic form fits well in the approach proposed here in the
sense that we recover 
the classical formula: 

\begin{lemma}\label{z5}
Let $Y \in \Gamma(E)$. Then
\[
\dive(Y)\chi_E=L_Y \chi_E
\]
where $L_Y$ is the Lie derivative.
\end{lemma}
\begin{proof}
Let $\{ E_1,\ldots, E_p\}$ be a local positively oriented
orthonormal basis of sections of $E$. 
Symmetry of the connection implies that
\[ 
g_E(L_YE_j,E_j)=-g_E(\nabla^E_{E_j}Y,E_j), 
\]
hence
\begin{eqnarray*}
L_Y\chi_E(E_1,\ldots, E_p)&=
&-\sum_{j=1}^p\chi_E(E_1,\ldots,L_YE_j,\ldots,E_p)\\
& =& -\sum_{j=1}^p g_E (L_Y E_j, E_j) \chi_E(E_1,\ldots,E_p) \\
&=
&\sum_{j=1}^pg_E (\nabla^E_{E_j}Y,E_j)\chi_E(E_1,\ldots, E_p) \\ &
= & \dive(Y)\chi_E(E_1,\ldots, E_p).
\end{eqnarray*}
\end{proof}

Again, using the terminology of Rumler \cite{Rumler}, we introduce the following

\begin{definition}
 Let $E^{\perp}$ be a orientable bundle and $q=n-p$. The
characteristic transverse form of $\mathcal{F}$, denoted by
$\chi_{E^{\perp}}$, is the differential $q$-form on $M$ defined by
\[
\chi_{E^{\perp}}(Y_1,\ldots,Y_q)=\det[g(Y_i,E_j)]
\]
where  $Y_1,\ldots,Y_q\in TM$ and $\{E_{p+1},\ldots,E_{n}\}$ is a
local positively oriented orthonormal basis of sections of
$E^\perp$.
\end{definition}

 The following identity holds:

\begin{lemma}\label{z6}
\begin{equation}\label{z7}
d\chi_{E^\perp}=-\kappa^{\flat} \wedge \chi_{E^\perp},
\end{equation}
where $\kappa=\pi_E\left(\sum_{j=p+1}^n\nabla_{E_j}E^j\right)$.
\end{lemma}

\begin{proof}
Let $\{E_{p+1},\ldots,E_{n}\}$ be a local positively oriented
orthonormal basis of sections of $E^\perp$. 

Initially, note that if  $Y\in E^\perp$, both sides of the Equation (\ref{z7}) 
vanish
 in the $(q+1)$-uple $(Y,E_{p+1},\ldots,E_{n} )$.
For $Y\in E$ we have that 

\[
g(L_Y E_j,E_j)=g(Y,\nabla_{E_j}E_j).
\]
Since $i_Y\chi_{E^\perp}=0$, by Cartan's formula:
\begin{eqnarray*}
d\chi_{E^\perp}(Y,E_{p+1},\ldots,E_{n})
&=&
L_Y\chi_{E^\perp}(E_{p+1},\ldots,E_{n})\\
&=&Y(\chi_{E^\perp}(E_{p+1},\ldots,E_{n}))\\
&&-\sum_{j=p+1}^n\chi_{E^\perp}(E_{p+1},\ldots,L_YE_j,\ldots,E_n)\\
&=&
 -g\left( Y,\sum_{j=p+1}^n\nabla_{E_j}E_j\right) \\
& = &- g(Y,\kappa).
\end{eqnarray*}
Thus
\[
d\chi_{E^\perp}(Y,E_{p+1},\ldots,E_{n})=-\sum_{j=p+1}^ng(Y,\pi_E\nabla_{E_j}
E_j)=-\kappa^\flat(Y),
\]
for all $Y\in TM$.\end{proof}

The vector field $\kappa$ defined in the lemma above can also be characterized
as the unique section
on $E$ such that
\begin{equation}
 \label{rmk4}
\kappa^{\flat}(X)=\dive(X)-\mathrm{div}(X),
\end{equation}
for all $X\in \Gamma E$. Yet, $\kappa$ is the trace on $E^\perp$ of the
bilinear form $b$, defined by
\[
b(V,W)=\pi \nabla_VW.
\]

The Laplacian $\Delta_M$ can be written in terms of the foliated
Laplacian $\Delta_E$, the section $\kappa$ and the basic Laplacian
$\Delta_b$. Let $\delta_b$ denote the formal adjoint of the exterior derivative
$d$
restricted to basic
forms and denote by $*$ the Hodge star operator (see e.g. Tondeur \cite[p.
134]{Tondeur}). The basic Laplacian $\Delta_bf=\delta_bdf$ is the second order
operator given by
\[
\Delta_bf=(-1)^{p(q-1)} *\left[ 
 K^{\flat}\wedge (*(df\wedge\chi_E) )\wedge\chi_E) - d (*(df\wedge\chi_E)
)\wedge\chi_E) \right].
\]

\begin{theorem}\label{z8}
If  $M$ is oriented, then 

\[
\Delta f=(\Delta_E f-\kappa f)+ (-1)^{pq+1}\Delta_bf.
\]
\end{theorem}

\begin{proof} Orientability of $M$ implies that $ \chi^\perp=*\chi$.
Initially, consider a vector field $Y$ in  $M$. We claim that
\begin{eqnarray}\label{a1}
i_{Y}\chi^\perp=(-1)^{p}*(Y^\flat\wedge\chi).
\end{eqnarray}
In fact, let
$\{V_1,\ldots,V_q,E_1,\ldots,E_p\}$ be an adapted orthonormal basis of $TM$ and
consider its dual basis
$\{V_1^\flat,\ldots,V_q^\flat,E_1^\flat,\ldots,E_p^\flat\}$. We have
\[
\chi=E_1^\flat\wedge\ldots\wedge E_p^\flat \ \ \ \ \mbox { and } \ \ \ \
\chi^\perp=V_1^\flat\wedge\ldots\wedge V_q^\flat.
\]
It is enough to prove the equality for $Y=V_j$, for $j=1, \ldots , q$. 
But, for a fixed $j$, Equation (\ref{a1}) follows from
\[
i_{V_j}\chi^\perp=(-1)^{j+1}(V_1^\flat\wedge\ldots\wedge\widehat{V_j^{\flat}}
\wedge\ldots\wedge V_q^\flat)
\]
and
\[
*(V_j^{\flat}
\wedge\chi)=(-1)^{p+j-1}(V_1^\flat\wedge\ldots\wedge\widehat{V_j^{\flat}}
\wedge\ldots\wedge V_q^\flat).
\]

Secondly, let $f$ be a smooth function in $M$, by Cartan formula,
\begin{eqnarray}
di_{(\grad f)}\ \chi^\perp & = &di_{(\pi^\perp \grad f)}\ \chi^\perp \nonumber
\\
&=&L_{(\pi^\perp \grad f)}\ \chi^\perp-i_{(\pi^\perp \grad f)}\ d\chi^{\perp},
\label{a2}
\end{eqnarray}
where $\pi^\perp:TM\rightarrow E^\perp$ is the orthogonal projection. Equation
(\ref{z7}) implies that,
\begin{eqnarray}\label{a3}
 i_{\pi^\perp\grad f}d\chi^\perp&=&k^\flat \wedge i_{\pi^\perp\grad
f}\chi^\perp\nonumber \\
&=&k^\flat \wedge i_{\grad f}\chi^\perp.
\end{eqnarray}
 Rumler formula  (Tondeur \cite[pg. 66]{Tondeur}) says that
\begin{eqnarray}\label{ar}
L_Z\chi|_E+K^\flat (Z)\chi|_E=0,\hspace{1cm}\forall Z\in E^\perp.
\end{eqnarray}
Hence, combining the above equations, we have that
\begin{eqnarray*}
(-1)^{p}d*(df\wedge\chi)\wedge\chi
&=& (di_{\grad f}\chi^\perp)\wedge\chi \hspace{2.5cm}\textrm{by }(\ref{a1})\\
&=& (L_{\pi^\perp\grad f}\chi^\perp)\wedge\chi\\
&&-k^\flat\wedge (i_{\grad f}\chi^\perp)\wedge\chi\hspace{1.6cm}\textrm{by
}(\ref{a2})\textrm{ and }(\ref{a3})\\
&=& L_{\pi^\perp\grad f}\mu_g\\
&&-\chi^\perp\wedge (L_{\pi^\perp\grad f}\chi)\hspace{.7cm}
\\
&=& L_{\pi^\perp\grad
f}\mu_g+ K(f)\mu_g, 
\end{eqnarray*}
by Equation (\ref{ar}).  Again, by Equation (\ref{a1})  and the fact that
$K^\flat \wedge\chi^\perp=0$
one can show that
\begin{eqnarray*}
(-1)^{p}K^\flat \wedge *(df\wedge\chi)\wedge\chi&=&K^\flat \wedge i_{\grad
f}\chi^\perp\wedge\chi\\
&=&-i_{\grad f}(K^\flat \wedge \chi^\perp)\wedge\chi+(i_{\grad f}K^\flat
)~\chi^\perp\wedge\chi\\
&=& K(f)\mu_g.
\end{eqnarray*}
Now, replacing this formula in the definition of $\Delta_b$, we
obtain
\[
\Delta_bf=
-(-1)^{pq} \mathrm{div} (\pi^\perp\grad f).
\]
Thus, decomposing $\grad f$ into tangential and normal components:
\[
\Delta f=
(\Delta_E f-\kappa f)+ (-1)^{pq+1}\Delta_bf.
\]
\end{proof}



In order to construct foliated Brownian motion in the next section, we have to
study the 
horizontal lift of the foliated Laplacian.
We consider the principal bundle $O(E)$ of orthonormal frames in $E$, with
projection
$r:O(E)\rightarrow M$ and structural group
$O(p)$.

\bigskip

 The induced connection $\nabla^E$ gives a partition of the tangent bundle of
$O(E)$ into a vertical space $VO(E)$ and a horizontal space
$HO(E)$ such that $TO(E)=VO(E) \oplus HO(E)$ (see e.g. Kobayashi and Nomizu
\cite{kobayashi nomizu}). 

For each  $v$ in $\mr^p$ the standard
vector field $H_v$ in $O(E)$, is given by the
unique $H_v(u) \in HO(E)_u$ such that $r_*(H_v(u))=uv$. For an orthonormal frame
$\{e_1,\ldots,e_p\}$  of
$\mathbf{R}^p$, we define the horizontal foliated Laplacian in
$O(E)$ as
\[
\Delta_E^H=\sum_{i=1}^p(H_{e_i})^2.
\]
One checks that it is independent of the basis.

\bigskip

The following lemma shows that
$\Delta_E^H$ is the horizontal lift of $\Delta_E$:

\begin{lemma}\label{z29} For $f\in C^\infty(M)$, the following identity holds
\[
\Delta_E^H(f\circ r)=(\Delta_Ef)\circ r.
\]
\end{lemma}

\begin{proof}  We first observe that
\begin{eqnarray*}
H_{e_j}(f\circ r)(u)&=&g\left(\grade f(r(u)),ue_j\right).
\end{eqnarray*}
For the second derivative, consider a horizontal curve $u_t$ in $O(E)$ such that
$u_0=u$ and $\dot
u_t=H_{e_i}(u_t)$. Then, for each $e_i$, the vector field $u_te_ i$ is
the parallel transport of $ue_i$ along $\gamma_t=r(u_t)$ with
respect to the connection $\nabla^E$. Hence,
\begin{eqnarray*}
H_{e_i}H_{e_j}(f\circ r)(u)&=&\left.\frac{d}{dt}\right|_{t=0}g\Big(\grade
f(r(u_t)),u_te_j\Big) \\
&=&
g\left(\nabla^E_{ue_i}\grade f(r(u)),ue_j\right) \\
&=&\textrm{Hess}_Ef(ue_i,ue_j)\circ r(u). \end{eqnarray*}

 So,
\begin{eqnarray*}
\Delta^H_E(f\circ
r)(u)&=&\left(\sum_{i=1}^p\textrm{Hess}_Ef(ue_i,ue_i)\right)\circ r(u).
\end{eqnarray*}
\end{proof}

\section{The foliated Brownian motion}

In this section we introduce the probabilistic aspects which are the key points
of our approach. We shall denote by
$(\Omega,\mathfrak{F},\{\mathfrak{F}_t\},\mP)$ a
filtered probability space
 satisfying the usual completeness conditions.
 

 A semimartingale $X$ in $M$  will be called
\emph{foliated} if each trajectory stays  in a single leaf.
Furthermore, a foliated semimartingale $X$ will be called a \emph{foliated
martingale} if for any
smooth function $f$,
\[
f(X)-f(X_0)-\frac{1}{2}\int_0 \textrm{Hess}_E f(dX,dX)
\]
is a local martingale. Foliated martingales may not be martingales in $M$;
precisely, this fact depends on the geometry of the foliation:

\begin{proposition}\label{z12}
 Foliated martingales are martingales in $M$ if and only if the foliation is
totally geodesic.
\end{proposition}

\begin{proof} A foliation is totally geodesic when  $\hess_E f(X,Y)= \hess
f(X,Y)$ for all $f\in C^\infty(M)$, hence the result follows.

Conversely, consider  $\gamma$ a geodesic of a leaf. We have to proof that
$\gamma$ is a geodesic in $M$.
For a linear Brownian motion  $B$, the process  
$X=\gamma(B)$ is a foliated martingale, hence, by hypothesis, it is a martingale
in $M$. By standard calculation
\begin{eqnarray*}\int\frac{d^2}{dt^2}f(\gamma)(B)dt
&=&\int\textrm{Hess}_Mf(dX,dX)\\
&=&\int (\gamma_*\otimes\gamma_*)^*\textrm{Hess}_Mf (dB,dB)\\
&=&\int \tr  (\gamma_*\otimes\gamma_*)^*\textrm{Hess}_Mf (B) dt\\
&=&\int\hess_M f(\dot\gamma(B),\dot\gamma(B))dt.
\end{eqnarray*}
It follows that 
\begin{eqnarray*}
\frac{d^2}{dt^2}f(\gamma)(t)&=&\textrm{Hess}_Mf(\dot\gamma(t),
\dot\gamma(t)),
\end{eqnarray*}
i.e. $\gamma$ is a geodesic in $M$.
\end{proof}

\vspace{.5cm}

Next result says
that
 the foliated martingales also satisfy the nonconfluence
property:
\begin{proposition}\label{z13}
For each $x \in M$ there is an open neighborhood $U_x \subset M$
such that, if $X$ and $Y$ are foliated martingales in $U_x$ such
that $X_T=Y_T$ for a stopping time $T$ then $X_t=Y_t$ a.e. for $0\leq t \leq T$.
\end{proposition}

\begin{proof}
The proof follows similar ideas for nonconfluence of martingales in a manifold
as in Emery \cite[p.52-53]{emery}. For a fixed point $p \in M$ we consider a
convex function
$f$ defined on a neighborhood $U \subset L_p\times L_p$ where
$L_p$ is the leaf through of $p$. By continuity, we extend this
function to $\widetilde{f}:\widetilde{U}\subset M \times M
\rightarrow \R$ such that Hess$_{E\times E}\widetilde{f}
(A,A)\geq 0$ for all $A\in E\times E$ and
$\widetilde{f}|\{(x,x)\in \widetilde{U} \}=0$. Clearly, there
exists $U_x$ neighborhood of $x$ such that $U_x \times U_x \subset
\widetilde{U}$. Let $X$ and $Y$ be foliated martingales in $U_x$
such that $X_T=Y_T$ a.e.. Using that $\widetilde{f}(X,Y)$ is a
positive bounded submartingale null at time $T$, we conclude that
$X_t=Y_t$ a.e. for $0\leq t \leq T$.
\end{proof}

Let $X$ be a foliated semimartingale. We says that $X$ is a \textit{foliated
Brownian
motion} (FoBM) if for any smooth function $f$,
\[
f(X)-f(X_0)-\frac{1}{2}\int_0\Delta_E f(X)\:dt
\]
is a local martingale. Note that a process $X$ is a FoBM if and
only if it is
a foliated martingale and for any smooth function $f$,
\begin{equation}\label{z18}
[f(X),f(X)]=\int_0 |\textrm{grad}_Ef(X)|^2dt.
\end{equation}






The geometry of the foliation determines probabilistic properties of FoBM:
\begin{proposition}\label{z19}
 FoBM are martingales in $M$ if and only if the leaves are minimal
submanifolds, i.e. the foliation is 
harmonic.
\end{proposition}
\begin{proof}
Let $X$ be a FoBM. By Equation (\ref{z3}) and the
definition, we have that for all smooth function $f$,
\begin{equation}\label{z20}
f(X)-f(X_0)-\frac{1}{2}\int_0 \hess f~(dX,dX)
-\frac{1}{2}\int_0 Kf(X)~dt
\end{equation}
is a local martingale. Since $X$ is a martingale in $M$,
\[
\int_0 Kf(X)~dt=0,
\]
hence $K=0$.

Conversely, from Equation (\ref{z20}) and $K=0$ we have that
\[
f(X)-f(X_0)-\frac{1}{2}\int_0 \textrm{Hess}f~(dX,dX)
\]
is a local martingale. Thus $X$ is a martingale in $M$.
\end{proof}



 Garnett \cite{Garnett} introduced
 foliated heat kernels via foliated
semigroups
of operators which depends strongly on the geometry of the foliation, (see also
Candel \cite{candel1}).
We can recover the same semigroup considering the semigroup associated to a
FoBM, provided one can guarantee the existence of this stochastic process.

Focusing in this direction, we present an intrinsic  construction of FoBM. Our
argument corresponds to
an adaptation to  foliated spaces
of the techniques of Eells-Elworthy-Malliavin, classically used to
 construct  Brownian motions in a Riemaniann manifold
 (see e.g. \cite{Elworthy}, \cite{IW} and references therein) .

\begin{theorem}\label{z30}
Let $u_t$ be the solution of the Stratonovich equation
\begin{eqnarray*}
d u_t&=&\sum_{i=1}^pH_{e_i}(u_t)\: \circ dB^i,
\end{eqnarray*}
where $(B^1,\ldots,B^p)$ is a Brownian motion on $\R^p$ with $u_0$ as initial
condition in $O(E)$.
Then $r(u_t)$ is a FoBM in $M$ starting at $r(u_0)$.
\end{theorem}

\begin{proof}  For any smooth function $f$ in $M$, applying Lemma \ref{z29}, we
have that
\begin{eqnarray*}
f(r(u))-f(r(u_0))
&=&\sum_{i=1}^p\int H_{e_i}(f\circ r)(u)\: dB^i+\\
&&\frac{1}{2}\int \sum_{i=1}^pH_{e_i}^2(f\circ r)(u)dt \\
&=&\sum_{i=1}^p\int H_{e_i}(f\circ r)(u)~
dB^i_t\\
&&+\frac{1}{2}\int \Delta_Ef (r(u))dt.
\end{eqnarray*}

\end{proof}

\noindent \textbf{Example 1:} Let $N$ and $L$ be two Riemannian manifolds.
 Consider the product $M=N\times L$ with the canonical foliation given
by 
$E= TL \subset TM= TN \oplus TL$. The foliated Laplacian $\Delta_E = \Delta_L$,
hence if $W$ is a Brownian
motion in $L$ then $B=(x_0,W)$ is a FoBM. 

\eop

\noindent \textbf{Example 2:} (Kronecker foliation) Consider the totally
geodesic foliation
of the plane $\R^2$ along 
lines parallel to the vector $(a,1)$. The process 
    \[
    B=\frac{(a,1)}{\sqrt{a^2+1}} ~W,
    \]
where $W$ is a linear Brownian motion is a foliated Brownian motion. Now,
denoting by $\mathbf{T}^2$
 the 2-torus $S^1 \times S^1\subset\mr^4$,
 let $\phi:\R^2 \rightarrow\mathbf{T}^2$ given by
\[
\phi(x,y)=(\cos(x),\sin(x),\cos(y),\sin(y)).
\]
The induced foliation by $\phi$ in $\mathbf{T}^2$ is called the Kronecker
foliation, Candel and Conlon \cite{candel}.
 We claim that $\phi(B)$
is a FoBM. In fact,
$E=\{\lambda (a\partial_x+\partial_y),\:\lambda\in\mr\}$ and
\[
\Delta_E=\frac{1}{a^2+1}(a^2\partial_x^2+2a\partial_{xy}^2+\partial_y^2),
\]
where $\partial_x=\phi_*(e_1)$ and $\partial_y=\phi_*(e_2)$.

For all smooth function $f$ in $\mathbf{T}^2$, 
\begin{eqnarray*}
    f(\phi(B_t))&=&f(\phi(x_0))+\int_0^ta(\partial_xf)(\phi(B_s))
+(\partial_yf)(\phi(B_s))dW_s\\
    &&+ \frac{1}{2}\int_0^t(\Delta_Ef)(\phi(B_s))ds.
    \end{eqnarray*}
 Hence  $\phi(B)$
is a foliated Brownian motion in the Kronecker foliation of the torus.

\eop

Next proposition shows that 1-dimensional foliations generated by
unitary vector fields have an easy construction of FoBM:

\begin{proposition}     \label{1-dim FoBM}   
Let $M$ be a foliated Riemannian manifold where the distribution $E$
is generated  by a smooth unitary vector field $Y$. If $\phi_t$ is the flow of 
 diffeomorphisms associated to $Y$ and $B$ is a linear Brownian motion, then 
 $\phi_{B}(x_0)$ is a FoBM starting at $x_0\in M$.
\end{proposition}

\begin{proof}
We have that $ \Delta_E = Y^2 $ and $\phi_t (x_0)$ is a geodesic in the leaf.
Hence $\phi_{B}(x_0)$ is a martingale, so that

\[
f(\phi_{B}(x_0))-f(\phi_{0}(x_0))-\frac{1}{2}\int
\Delta_Ef(\phi_B(x_0))~dt
\]
is a local martingale for any smooth function $f$. So, $\phi_B(x_0)$ is a FoBM
starting at
$x_0$.
\end{proof}

\section{The Harmonic Measures}

In this section we focus
on the theory of harmonic measures, according to  Garnett \cite{Garnett}, Candel
\cite{candel1} and others. Associated to the FoBm we have $\{T_t\}$ the Markov
semigroup, determinated by
\[
 T_tf(x)=\me[f(B_t^x)],
\]
 acting in the space $\mathcal{B}_b^L$ of bounded measurable function $f$ which
are leafwise smooth. 

As a corollary of the It\^o formula we obtain that the infinitesimal generator
of $\{T_t\}$ is $1/2 \Delta_E$. A probability measure $\mu$ in $M$ is called
invariant under $\{T_t\}$ if
\[
 \int_MT_tf~d\mu=\int_Mf~d\mu
\]
for all $f\in \mathcal{B}_b^L$ and $t\geq0$.


We recall that a probability measure $\mu$ on a foliated manifold $M$ is 
called 
\textit{harmonic} if for any $f\in\mathcal{B}_b^L$
\[
\int_M\Delta_Ef ~ d \mu=0.
\]

\begin{lemma} A probability measure $\mu$ on $M$ is harmonic if and only if it
is
an invariant measure for $\{T_t\}$. \label{harmonico_e_invariant}
\end{lemma}
\begin{proof} We first observe that for any $f\in\mathcal{B}_b^L$ we have
that $T_tf\in\mathcal{B}_b^L$ and 
\[
 T_tf-f=\frac{1}{2}\int_0^t\Delta_E(T_sf)~ds~ ~~\in\mathcal{B}_b^L.
\]

If $\mu$ is harmonic then, by Fubini's theorem
\[
\int_M(T_tf-f)~ d \mu= \int_0^t\int_M\Delta_E(T_sf)~ d \mu~ds = 0.
\]
Now assume that $\mu$ is invariant.  By Fatou's lemma, we have
\[ 
\int_M\Delta f~d \mu \leq \lim_{t\rightarrow 0}\frac{1}{t}\int_M(T_tf-f)~d \mu
= 0.
\]
The same calculation for $-f$ implies that $ \int_M\Delta f\geq 0$.


\end{proof}

As a consequence of the support theorem (see e.g. Ikeda-Watanabe \cite{IW}),
 it follows that the support 
of a harmonic measure is a saturated set, i.e. it consists in a union of leaves.
Next theorem extends  results of Garnett \cite{Garnett} on existence of harmonic
measures.


\begin{theorem} \label{z31}Let $M$ be a foliated Riemannian
manifold
\begin{itemize}
\item[1)] If $M$ is compact then there exist  harmonic probability measures;
\item[2)] If the leaves are stochastically complete and there exists a smooth
function
 $\varphi \geq 0$ on $M$ such that $$\lim_{d(x,x_0)\rightarrow \infty} \Delta_E
\varphi(x) = - \infty, $$ then there exist 
harmonic probability measures.
 \item[3)]  If $f\in\mathcal{B}_b^L$  such that $\Delta_E f\equiv 0$ then   $f$
is
constant in the leaves of the support of any harmonic probability measure.
\end{itemize}
\end{theorem}

\begin{proof} Item (1) follows directly from Lemma \ref{harmonico_e_invariant}
and the existence of invariant
measure of diffusions in compact manifolds.


\bigskip

Item (2) is a consequense of Khas'minskii criterium for existence of invariant
measures, 
see e.g. Lorenzi and Bertoldi \cite[p.172]{Lorenzi-Bertoldi}.

\bigskip

For (3), let $f:M\rightarrow\mr$ be a leafwise harmonic
function ($\Delta_Ef=0$) and $\mu$ a harmonic measure. Using Equation
(\ref{z4.1}) one finds 
\[
\int_M |\grade f|^2(x)\ \mu(dx)=
\frac{1}{2}\int_M \Delta_E f^2(x)\ \mu(dx)     =0 .
\]  
Then, by continuity of $ |\grade f|$,  $f$ is
constant in the leaves of the support of $\mu$.

\end{proof}

\noindent \textit{Remark:} Stochastic completeness in item (2) of the
above Theorem can be
 guaranteed by well known geometrical conditions on the leaves, say, for example
if 
the leaves are complete and 
\[
 \int_{c>0} ^{\infty } \frac{r ~dr}{\ln |B(r)|} = \infty,
\]
where $|B(r)|$ is the volume of the geodesic ball of radius $r$, see Grigorian
\cite[Thm. 9.1, p.184]{Grigorian2} or yet, in terms of
curvature, if we have
lower bounds in the 
Ricci curvature, Elworthy \cite{Elworthy}. 

\bigskip




Given  $M$ and $M'$ two
foliated Riemannian manifolds, we say that a smooth map $\phi:M\rightarrow M' $ 
 is a foliated map if it preserves leaves. 
If also $\phi_* \Delta_E = \Delta_E'$, and $\mu$ is a harmonic measure in $M$,
note
that 
the induced measure $\phi_* \mu$ is a harmonic measure in $M'$.

\bigskip

Our formalism allows a direct proof of the result on superharmonic
functions on foliations in Adams \cite{adams}:

\begin{theorem}[Adams] \label{z32}
Let $\mu$ be a harmonic probability measure.
Consider $f:M\rightarrow(0,\infty)$ a bounded measurable function such that, for
$\mu$-a.s.
$\Delta_Ef\leq0$. Then $f$ is constant in each leaf $\mu$-a.s.
\end{theorem}

\begin{proof} Using the fact that for any smooth function $\phi$ we have
that $$\Delta_E (\phi \circ f)= \phi'(f(x))  \Delta_E f(x) + \phi''(f(x))
|\grade f (x)|^2$$ then  for  $u=\ln(f+1)$,  $u$ is
positive, belongs to $\mathcal{B}_b^L$ and satisfies
\begin{eqnarray}\label{eq en u}
|\grade u|^2+\Delta_Eu~ =~\frac{\Delta_Ef}{f+1}~\leq ~0,
\end{eqnarray}
on $\mu$-a.s., thus 
\[
 \int_M|\grade u|^2~ d \mu\leq0,
\]
hence $\grade f = (1 + f)~ \grade u$ vanishes.

\end{proof}

\noindent \textbf{Example 3:} 
Consider a torus $\mathbf{T}^2 \subset \mr^3$ immersed isometrically with
the covering coordinate system
$\phi:\R^2\rightarrow
\mathbf{T}^2$ given by
\[
\phi(x,y)=\big( (b+\cos(x))\cos(y),(b+\cos(x))\sin(y),\sin(x) \big),
\]
with $b>1$. In these coordinates the induced metric is
\[
g=dx^2+(b+\cos(x))^2dy^2,
\]
and the associated Riemannian connection $\nabla$ is characterised by
\[
\nabla_{\partial_x}\partial_x=0,\hspace{.5cm}\nabla_{\partial_x}\partial_y
=\frac{-\sin(x)}{b+\cos(x)}\partial_y,
\hspace{.5cm}\nabla_{\partial_y}\partial_y={(b+\cos(x))}{\sin(x)}\partial_x.
\]

Consider the foliation $E$ on $\mathbf{T}^2$ generated by 
\[
Y=\frac{1}{\sqrt{\alpha^2+1}}\left(\alpha\partial
_x+\frac{1}{(b+\cos(x))}\partial_y\right).
\]
The leaf through $(x_0,y_0)$ is the flow line
of  $Y$ through $(x_0,y_0)$, where the flow $\psi$ of $Y$  can be
represented in local coordinates as
$\psi_t(x_0,y_0)=\left(x_t\:,y_t\right)$ with
\begin{eqnarray*}
x_t&=&x_0+\frac{\alpha\:t}{\sqrt{1+\alpha^2}}\hspace{6.7cm}   (\mbox{mod}\ 
2\pi), \\
y_t&=&y_0+A+
\frac{2}{\alpha\:\sqrt{b^2-1}}\arctan\left(\sqrt{\frac{b-1}{b+1}}
\tan\left(\frac{x_t}{2}\right)\right)\hspace{.6cm} (\mbox{mod}\ 
2\pi), \\
A&=&\frac{-2}{\alpha\:\sqrt{b^2-1}}\arctan\left(\sqrt{\frac{b-1}{b+1}}
\tan\left(\frac{x_0}{2}\right)\right)\hspace{2.3cm} (\mbox{mod}\ 
2\pi).
\end{eqnarray*}
We observe that
\[
\nabla_YY=\frac{\sin(x)}{(1+\alpha^2)(b+\cos(x))}\:\partial_x-
\frac{\alpha\sin(x)}{(1+\alpha^2)(b+\cos(x))^2}\:\partial_y.
\]
Then $\nabla_Y^EY=0$ and $\Delta_E=Y^2$. So, a FoBM in
$\mathbf{T}^2$ is a solution of the  stochastic differential
equation
\begin{eqnarray*}
d W&=&Y(W)~\circ d B\\
W_0&=&p_0\in M
\end{eqnarray*} with $B$ the Brownian motion in $\mathbf{R}$. Therefore the FoBM
can be written in coordinates as the solution of
\begin{eqnarray*}
dX_t&=& \frac{\alpha}{\sqrt{1+\alpha^2}}\:dB_t, \\
dY_t&=&\frac{1}{(b+\cos(X_t))\sqrt{1+\alpha^2}}\:dB_t
+\frac{1}{2}\frac{\alpha\sin(X_t)}{(1+\alpha^2)(b+\cos(X_t))^2}\:dt.
\end{eqnarray*}
So,
\begin{eqnarray*}
X_t&=&x_0+\frac{\alpha}{\sqrt{1+\alpha^2}}\:B_t\hspace{1cm}{(\mbox{mod}\ 
2\pi)} ,\\
Y_t&=&y_0+A+\frac{2}{\alpha\:\sqrt{b^2-1}}\arctan
\left(\sqrt{\frac{b-1}{b+1}}\tan\left(\frac{X_t}{2}\right)\right)\\
&&\hspace{8cm}(\mbox{mod}\ 
2\pi).
\end{eqnarray*}

Therefore, cf. Proposition \ref{1-dim FoBM}, the FoBM starting at $(x_0,y_0)$ is
given by $
W_t=\psi_{B_t}(x_0,y_0)$. A measure $\mu=h\:\mu_g$ is harmonic for a  smooth
function $h$ 
 if and only if  $h$ satisfies
\[
Y^2(h)+2~\textrm{div}(Y)Y(h)+\left( Y(\textrm{div}(Y))+\textrm{div}(Y)^2 \right)
h=0.
\]
Considering the case of $h$ depending only on $x$, equation above reduces to 
\[
(b+\cos(x))h''(x)-2\sin(x) h'(x)-\cos(x)h(x)=0,
\]
whose unique non-trivial normalized periodic solution is 
\[
h(x)=\frac{1}{4\pi^2}\frac{1}{(b+\cos(x))}.
\]

\eop






In a foliated space $M$ with orientable leaves that admit a holonomy invariant
measure $ \nu$, a harmonic measure can be constructed in terms of $\nu$. The 
$p$-current
 $\varphi_\nu$ associated to $\nu$ is the functional in $\Lambda^p (M)$ given by
\[
\varphi_\nu (\omega) = \sum_{\alpha \in \mathcal{U}}~\int_{S_\alpha} \left(
\int_P \lambda_\alpha \omega \right) ~d\nu (P)
\]
where $\lambda_\alpha$ is a partition of unity subordinated to a foliated
atlas $\mathcal{U}$, $P$ are plaques in $U_\alpha\in\mathcal{U}$ and $S_\alpha$
is transversal in $U_\alpha$ (see Plante \cite[p.330]{plante} and Candel
\cite[p.235]{candel1}). 


The measure $\mu_{\nu} $ associated to the positive functional
$f\mapsto \varphi_{\nu}(f\chi_E)$ is called in the literature a 
\textit{ totally  invariant measure} (e.g. \cite{plante}). Such associated
measures have a further characterization which generalizes similar result in
\cite{Garnett}:

\begin{theorem} Let $M$ be a compact foliated Riemannian manifold leafwise
orientable.
A measure $\mu$ is totally invariant if and only if 
\[
 \int_M \dive X ~d\mu=0
\]
 for any $X\in \Gamma(E)$.

\end{theorem}

\begin{proof}  Let $\mu$ be a totally invariant measure associated to $\nu$.
We have to prove that
\[
 \varphi_{\nu}( \dive X \chi_E )=0.
\]
 But 
\[
   \dive (X) \chi_E = di_X \chi_E + i_X d\chi_E,
 \]
 $\varphi_{\nu} (d \alpha) =0$ and 
$i_{X}\left( d\chi_E \right)=0$ restricted to the leaves 
(cf. \cite[p.69]{Tondeur}).

For the converse, note that $\mu$ is harmonic. There exists a  foliated atlas
$\{U_i\simeq T_i\times P\}$ and an associated  family of leafwise positive
harmonic functions $h_i: U_i \rightarrow \R$ with corresponding transverse
measures $\nu_i$  such that for any measurable function $f$,   
\[
\int_{U_i}  f ~d\mu= \int_{T_i} \int_{\{t\}\times D_t} f \chi_E ~d\nu_i
\]
see Garnett \cite[Theorem 1-c]{Garnett} or Candel and Conlon \cite[Vol II,
Prop.2.4.10]{candel}. We take a partition of unity $\{\lambda_i\}$ subordinated
to the foliated atlas  $\{U_i\simeq T_i\times P\}$. Hence

\begin{eqnarray*}
  \int_M\dive(X)~d\mu&=& \sum_i \int_M\dive(\lambda_i X)\mu\\
&=&\sum_i \int_{T_i}\left(\int_{\{t\}\times D_t}\dive(\lambda_i
X)h_i~\chi_E(t)\right)d \nu_i(t)\\
&=&\sum_i \int_{T_i}\left(\int_{\{t\}\times D_t}\dive(h_i\lambda_i
X)~\chi_E(t)\right)d \nu_i(t)\\
&&-\sum_i \int_{T_i}\left(\int_{\{t\}\times D_t}\lambda_i X
(h_i)~\chi_E(t)\right)d \nu_i(t)\\
&=&-\sum_i \int_{T_i}\left(\int_{\{t\}\times D_t}\lambda_i X
(h_i)~\chi_E(t)\right)d \nu_i(t)\\
\end{eqnarray*}
where each $\nu_i=p_*(\mu|_{U_i})$ is a measure over $T_i$ induced by the
projection $p:U_i\rightarrow T_i$. Thus, $ \int_M\dive(X)\mu=0$ for any $X$ if
and only if $h_i$ is constant along the leaf, that is $\mu$ is associated to the
transverse measure $\nu= \sum_i h_i \nu_i$ which is invariant under holonomy
transformations.

\end{proof}


\begin{corollary}
 Let $M$ be a compact foliated Riemannian manifold leafwise orientable. A
measure $\mu$ is totally invariant if and only if 
\[
 \int_M\textrm{div}(X)~\mu=-\int_M \kappa^{\flat}(X)~\mu
\]
for all $X\in\Gamma(E)$.
\end{corollary}
\begin{proof}
 It follows directly from Equation (\ref{rmk4}).
\end{proof}


%


\begin{proposition}
Let $M$ be a compact foliated Riemannian manifold leafwise orientable
 with $\nu$  a holonomy invariant measure.  If $\mu $ and  $\widetilde{\mu} $
are two harmonic
measures such that $\tilde \mu$ has a Radon-Nikodym derivative
$h$ belonging to $\mathcal{B}_b^L$ with respect to $\mu$, then $h$ is constant
in the leaf ($\mu$-a.e).
\end{proposition}

\begin{proof} 



We have by Equation (\ref{z4.1}),
\[
\int_M|\grade h|^2\mu = \frac{1}{2}\int
\Delta_Eh^2~\mu
-\int_M(\Delta_Eh)~\tilde\mu = 0.
\]

\end{proof}


Harmonic measures which are absolutely continuous with respect to 
the Riemanian volume $\mu_g$ are characterized in the following:

\begin{theorem}
Let $M$ be a compact foliated Riemannian manifold
without boundary and $h$ be a non negative function which is
in $\mathcal{B}_b^L$. Then $h\mu_g$ is harmonic if and only if $h$
satisfies
\begin{equation} \label{z33}
\mathrm{div}(\grade h-h\kappa)=0\hspace{1cm}\mu_g-\textrm{a.s.}
\end{equation}
\end{theorem}

\begin{proof}
Firstly, we claim that the operator $\Delta_E-\kappa$ is self-adjoint. 
In fact, by Equation (\ref{rmk4}) we have that
\[
(\Delta_Ef-\kappa f)\mu_g=\textrm{div}(\grade f)\mu_g. 
\]
Using that 
\[
 \mathrm{div} \left( h ~ \grade f \right) = g(\grade f, \grade h) + h
~\textrm{div}(\grade f),
\]
one finds that
\begin{eqnarray*}
\int_Mh(\Delta_E-\kappa)f\mu_g&=&\int_M h ~ \textrm{div}(\grade f)\mu_g \\
&=&-\int_Mg(\grade f,\grade h)\mu_g \\
&=&\int_M f ~ \textrm{div}(\grade h)\mu_g \\
&=&\int_M f(\Delta_E-\kappa)h\mu_g.
\end{eqnarray*}
For any smooth function $f$ we have that 
\[
 \mathrm{div} (fh\kappa)= h \kappa (f) + f \kappa (h) + fh \mathrm{div}(\kappa).
\]
Hence, 
\begin{eqnarray*}
 \int_M f \mathrm{div}(\grade h-h\kappa)~ \mu_g & =&  \int_Mf(\Delta_Eh-2\kappa
h-\textrm{div}(\kappa)h)\mu_g \\
   & = & \int_Mf(\Delta_Eh-\kappa h)\mu_g+\int_M(\kappa f)h\mu_g\\
&=&\int_M(\Delta_Ef)h\mu_g
\end{eqnarray*}
which vanishes for any $f$ if and only if $h\mu_g$ is harmonic.

\end{proof}

\begin{corollary} \label{div-zero}

Let $M$ be a compact foliated Riemannian manifold
without boundary. Then $ \mathrm{div}(\kappa)=0$ if
and only if for every bounded non-negative leafwise
constant function $h$ the measure $\mu=h\mu_g$ is harmonic.
\end{corollary}

\begin{corollary}
Let $M$ be a compact foliated Riemannian manifold
without boundary. Then
\[
\int_M|\grade h|^2 ~d\mu_g = -\frac{1}{2}\int_Mh^2\textrm{div}(\kappa) ~\mu_g ,
\]
for any bounded leafwise harmonic function $h$.
\end{corollary}

\begin{proof}
One uses that  $\Delta_E-\kappa$ is  selfadjoint, Equation (\ref{z4.1}) and
the Gauss theorem to find that
\[
0=\int_M(\Delta_E-\kappa)h^2\mu_g=2\int_M|\grade h|^2\mu_g-\int_M\kappa
(h^2)\mu_g.
\]
\end{proof}







\noindent \textbf{Example 4:}
Let $M$ be a quotient of the
universal covering of $Sl(2,\mr)$ by a cocompact lattice. Denote by
 $\{X,Y,H\}$ an orthonormal basis of $TM$ satisfying
\[
[X,H]=X\hspace{1cm}[X,Y]=-H\hspace{1cm}[H,Y]=Y.
\]
Consider the foliation induced by $E=\mathrm{span} \{X,H\}$. 
We have that 
\[
\Delta_E=X^2+H^2+H.
\]
Hence FoBM satisfies the following stochastic differential equation
\[
dB=Hdt+H\circ d B^1+X\circ d B^2
\]
where $(B^1,B^2)$ is the Brownian motion in $\mr^2$. In this case we also have
that $\kappa=H$ and $\mathrm{div}(\kappa)=0$
which implies that the volume measure $\mu_g$ and any $h \mu_g$ with $h$
constant in the leaves are harmonic (cf. Corollary \ref{div-zero}).
But any smooth $h$ which is leafwise constant is constant in $M$: In fact, note
that for any smooth function $f$,
\begin{equation} \label{estrela}
 \int_M Hf ~\mu_g = \int_M L_H (f\mu_g) - \int_M f \mathrm{div} H \mu_g
\end{equation}
and both terms on the right hand side vanishes by Cartan formula and Stokes
theorem. Moreover, we observe that
\begin{eqnarray*}
H(Yh)&=&Yh+YHh \\
&&=Yh.
\end{eqnarray*}
Hence, applying Equation (\ref{estrela}) to $(Yh)^2$,
\begin{eqnarray*}
0=\int_M H(Yh)^2\mu_g&=&2\int_M(Yh)H(Yh)\mu_g \\
&=&2\int_M(Yh)^2\mu_g.
\end{eqnarray*}
Then $h$ is constant. 


\eop

\noindent \textbf{Example 5:} (Lie foliations) Let $M$  be a manifold and
$\mathfrak{g}$ a Lie algebra of
dimension $q$. Assume that there exists a non singular surjective
$\mathfrak{g}$-valued 1-form
$\theta$ which satisfies the Maurer-Cartan formula
\[
d\theta+\frac{1}{2}[\theta,\theta]=0.
\]
Consider the Lie foliation $E= \ker \theta_x$.

%

Let $Y_1,\ldots, Y_q$ be vector fields in $TM$ such
that $\theta (Y_i)$, $i=1. \ldots, q $ is an orthonormal basis in
$\mathfrak{g}$. 
We introduce an adapted metric on $M$,
i.e., a metric $g$ such that $g(Y_k,Y_j)=\delta_{kj}$ and $
Y_1,\ldots, Y_q$ is an orthonormal basis of $E^\perp$. Hence, for all $X\in E$,
\[
g(\kappa,X)=\sum_{k=1}^qg(\nabla_{Y_k}Y_k,X)=g([Y_k,X],Y_k)=0.
\]
Then $\kappa=0$ therefore the volume measure $\mu_g$ and any $h \mu_g$ with $h$
constant in the leaves are harmonic (cf. Corollary \ref{div-zero}).

\eop

\end{document}